\documentclass[a4paper,11pt]{amsart}

\usepackage{amsfonts,amssymb}
\usepackage{textcomp}

\theoremstyle{plain}

\newtheorem{thm}{Theorem}[section]
\newtheorem{dfn}[thm]{Definition}

\newtheorem{lem}[thm]{Lemma}
\newtheorem{cor}[thm]{Corollary}

\newtheorem{rmk}[thm]{Remark}

\DeclareMathOperator{\C}{c}

\begin{document}

\title[ Non-inner automorphisms ]%
{ Non-inner automorphisms of order $p$ \\ in finite  $p$-groups of coclass 3 }

\author[M.Ruscitti]{Marco Ruscitti}
\address{DISIM \\Universit\`a degli studi dell'Aquila\\ 67100 L'Aquila, Italy\\
{\it E-mail address}: { \tt marco.ruscitti@dm.univaq.it }}

\author[L.Legarreta]{Leire Legarreta}
\address{Matematika Saila\\ Euskal Herriko Unibertsitatea UPV/EHU\\
48080 Bilbao, Spain\\ {\it E-mail address}: {\tt leire.legarreta@ehu.eus}}

\author{manoj K. Yadav}
\address{Harish-Chandra Research Institute, Chhatnag Road, Jhunsi
Allahabad, 211 019, India  \\
{ \it E-mail address}: {\tt  myadav@hri.res.in}}

\thanks{The first author would like to thank the Department of Mathematics at the University of the Basque Country for its excellent hospitality while part of this paper was being written. The second author is supported by the Spanish Government, grants MTM2011-28229-C02-02 and MTM2014-53810-C2-2-P, and by the Basque Government, grant IT753-13. The third author thanks INDAM and the Department of Mathematics at the University of L'Aquila for its excellent hospitality for a month during June - July, 2015.
}

\keywords{Finite p-groups, non-inner automorphisms, derivations, coclass\vspace{3pt}}

\begin{abstract}
In this paper we study the existence of at least one non-inner automorphism of order $p$ of a non-abelian finite $p$-group of coclass $3$, whenever the prime $p\neq 3$.  
\end{abstract}

\maketitle

\section{Introduction}

The motive of this paper is to contribute to the following longstanding conjecture of Berkovich \cite[Problem 4.13]{khukhro:2010}, posed in 1973:

\vspace{.1in}

\noindent{\bf Conjecture.}  Every finite $p$-group admits a non-inner automorphism of order $p$, where $p$ denotes a prime number. 

\vspace{.1in}

This conjecture, which will be called {\it the conjecture} throughout this paper, can be viewed as a refinement of the following celebrated theorem of Gasch\"{u}tz \cite{gaschutz:1966}: Every non-abelian finite $p$-group admits a non-inner automorphism of order some power of $p$. 

The conjecture  has attracted the attention of many mathematicians during the last couple of decades, and has been confirmed for many interesting classes of finite $p$-groups.  It is interesting to put on record that, in  1965, Liebeck \cite{liebeck:1965} proved the existence of a non-inner automorphism of order $p$ in all finite $p$-groups of class $2$, where $p$ is an odd prime. For $p=2$, he proved the existence of a non-inner automorphism of order $2$ or $4$. The fact that there always exists a non-inner automorphism of order $2$ in all finite $2$-groups of class $2$ was proved by Abdollahi \cite{abdollahi:2007} in 2007. The conjecture was confirmed for finite regular $p$-groups by Schmid \cite{schmid:1980} in 1980. Deaconescu \cite{deaconescu:2002} proved it for all finite $p$-groups $G$ which are not strongly Frattinian, in other words, groups satisfying $C_G(Z(\Phi(G))) \ne \Phi(G)$. Abdollahi \cite{abdollahi:2010} proved it for finite $p$-groups $G$ such that $G/Z(G)$
  is a powerful $p$-group, and Jamali and Viseh \cite{jamali:2013} proved the conjecture for finite $p$-groups with cyclic commutator subgroup. In the realm of finite groups, quite recently, the conjecture has been confirmed for $p$-groups of nilpotency class $3$, by Abdollahi, Ghoraishi and Wilkens  \cite{abdollahi':2013},  and  for  $p$-groups of coclass $2$ by Abdollahi et al \cite{abdollahi:2014}.  Finally,  for semi-abelian $p$-groups, the conjecture has been confirmed by Benmoussa and  Guerboussa \cite{benmoussa:2015}.

In this paper we add an important class of $p$-groups to the above list by proving the following result.

\vspace{.1in}

\noindent {\bf Main Theorem.} {\it The  conjecture holds true for all non-abelian finite $p$-groups of coclass $3$, where $p$ is a prime integer such that  $p\neq 3$.}

\vspace{.1in}

Since the conjecture  holds true for all finite $p$-groups $G$ having nilpotency class at most $3$ (\cite{liebeck:1965, abdollahi:2007, abdollahi':2013}), from now on,  we can assume that the nilpotency class of the  groups under consideration is greater than $3$. 

The organization of the paper is as follows. In Section 2 we put on record some preliminary results  which will  be used in the subsequent sections and state some structure theorems. Certain derivations are introduced, which will be used for constructing automorphisms of $p$-groups. Section 3, mainly, deals with useful tools and processes for constructing automorphisms of order $p$ of $p$-groups. Finally, the proof of Main Theorem is presented in the last section.

We conclude this  section by setting some notations which are mostly standard. For a given finite $p$-group $G$ and an integer $k \ge 1$, we denote by $\gamma_k(G)$ and $Z_k(G)$, respectively, the $k$-th term of the lower and the upper central series of $G$. The nilpotency class of a nilpotent group $G$ is denoted by $c(G)$. Recall that a finite group $G$ of order $p^n$ is said to be of \emph{coclass} $cc(G)$ if $c(G) = n - cc(G)$. The number of elements in any minimal generating set of a finite $p$-group $G$ is denoted by $d(G)$. Indeed, for a subgroup $H$ of a group $G$, $C_G(H)$ denotes the centralizer of $H$ in $G$, and analogously, for an element $x$ in a group $G$, $C_G(x)$ denotes the centralizer of $x$ in $G$. Finally, $\Phi(G)$ denotes the Frattini subgroup of the group $G$.

\section{Preliminaries and well-known results}
In this section, we point out some preliminaries and recall some well-known results that will be useful in the subsequent sections. We start with the following  two results due to Abdollahi  \cite{abdollahi:2010}, which allow us to understand  the connection between the existence of non-inner automorphisms of order $p$ and the structure of the first two terms of the upper central series in a finite $p$-group. 

\begin{lem} \label{Neq}
Let $G$ be a finite $p$-group such that $G$ has no non-inner automorphisms of order $p$ leaving $\Phi(G)$ elementwise fixed. Then, 
$$\Omega_{1}(Z(Inn(G))) \cong \Omega_{1}(Z_{2}(G)/Z(G)) \cong \underbrace{\Omega_{1}(Z(G)) \times \cdots \times \Omega_{1}(Z(G))}_{d(G) \ times}.$$
\end{lem}

\begin{cor}\label{d2}
Let $G$ be a finite $p$-group such that $G$ has no non-inner automorphisms of order $p$ leaving $\Phi(G)$ elementwise fixed. Then, 
$$d(Z_{2}(G)/Z(G))=d(G)d(Z(G)).$$
\end{cor} 

Next we point out some facts about derivations in the additive setting. The reader is referred to  \cite{gavioli:1999} for more details and explicit proofs.

\begin{dfn}
Let $G$ be a group and  $M$ be a right $G$-module. A derivation $\delta:G \rightarrow M$ is a function such that $$\delta(gh)=\delta(g)h + \delta(h)  \text    { for all } g,h \in G.$$
\end{dfn}

As a remark, in the multiplicative setting, the derivation $\delta$ is defined by the rule $\delta(gh)={\delta(g)}^h\delta(h)$   for all  $g, h \in G$.

In terms of its properties, let us note that a derivation is uniquely determined by its values over a set of generators of $G$. Let $F$ be a free group generated by a finite subset $X$ and let $G=\langle X : r_{1}, \ldots, r_{n} \rangle$ be a group whose free presentation is $F/R$, where $R$ is the normal closure of the set of relations $\{ r_{1} , \ldots , r_{n}\}$ of $G$. Then a standard argument shows that $M$ is a $G$-module if and only if $M$ is an $F$-module on which $R$ acts trivially. Indeed, if we denote by $\pi$ the canonical homomorphism $\pi: F \rightarrow G$, then the action of $F$ on $M$ is given by $mf=m\pi(f)$, for all $m\in M$ and all $f\in F$. On the other hand, as it has already been said, every derivation $\delta :F \rightarrow M$ is uniquely determined by the assignments on the generators of $F$, and taking into account \cite{gruenberg:1970}, we can note that free groups are suitable places for constructing derivations.

Continuing with the same notation, we have the following results.

\begin{lem} \label{determineunique}
Let $M$ be an $F$-module. Then every function $f: X \rightarrow M$ extends in a unique way to a derivation $\delta: F \rightarrow M$.
\end{lem}

\begin{lem} \label{Der}
Let $M$ be a $G$-module and  $\delta: G \rightarrow M$ be a derivation. Then $\bar{\delta} : F \to M$ given by the composition $\overline{\delta}(f)=\delta(\pi(f))$ is a derivation  such that $\overline{\delta}(r_i)=0$ for all $i\in \{1,\ldots,n\}$. Conversely, if $\overline{\delta}:F \rightarrow M$ is a derivation such that $\overline{\delta}(r_i)=0$ for all $i \in \{1,\ldots, n\}$, then $\delta(fR)=\overline{\delta}(f)$ defines, uniquely, a derivation on $G=F/R$ to $M$ such that $\overline{\delta}=\delta \circ \pi$.
\end{lem}

In the following lemma we study a relationship between derivations and automorphisms of a finite $p$-group. In particular, if $M$ is a normal abelian subgroup of a finite $p$-group $G$,  it follows that $M$ has a $G$-module structure, and as well, an $F$-module structure, where $F/R$ is a free presentation of $G$.  

\begin{lem} \label{lift}
Let $G$ be a finite $p$-group and  $M$ be a normal abelian subgroup of $G$ viewed as  a $G$-module. Then for any derivation $\delta: G \rightarrow M$, we can define uniquely an endomorphism $\phi$ of $G$ such that $\phi(g)=g\delta (g)$ for all $g \in G$. Furthermore, if  $\delta(M)=1$,  then $\phi$ is  an automorphism of $G$.
\end{lem}

Now we  exhibit two results that  allows us to simplify some computations  in the subsequent sections. 

\begin{lem} \label{free}
Let $F$ be a free group,  $p$ be a prime number and  $A$ be an $F$-module. Let $\delta : F \rightarrow A$ be a derivation. Then, 
\begin{enumerate}
\item  $\delta(F^{p})=\delta(F)^{p}[\delta(F),\ _{p-1}F]$,
\item if $A \leq F$, $[A,\ _{i}F]=1$, we have $\delta(\gamma_{i}(F)) \leq [\delta(F),\ _{i-1}F]$ for all $ i \in \mathbb{N}$.
\end{enumerate}
\end{lem}
\begin{proof}
Let $ x \in F$. We have $\delta(x^{p})=\delta(x)^{x^{p-1} + x^{p-2} + \cdots + 1 }$. Since $(x-1)^{p-1} \equiv  x^{p-1} + x^{p-2} + \cdots + 1\mod p$, the first assertion follows. Now we show the second assertion by induction on $i$.
 Clearly, the assertion holds when $i=1$. 
For inductive hypothesis, we assume that if $[A,\ _{k}F]=1$,
then $\delta(\gamma_{k}(F)) \leq [\delta(F),\ _{k-1}F]$ for some $ k \in \mathbb{N}$. 
Take any $a \in F$ and any $b \in \gamma_{k}(F)$, and  suppose 
that $[A,\ _{k+1}F]=1$. Then, $$\delta([a,b])=[\delta(a),b][a,\delta(b)][a,b,\delta(a)][a,b,\delta(b)] \in 
  [\delta(F),\ _{k}F].$$
\end{proof}

\begin{lem} \label{powers}
Let $G$ be a finite $p$-group and  $A$ be a normal abelian subgroup of $G$. Let  $\delta$ be a derivation from $G$ to $A$ and  $\phi$ be an endomorphism of $G$ defined by the law $\phi(g)=g\delta(g)$, for all $g \in G$. Then the following formula holds,
 $$\phi^{i}(g)=\prod_{j=0}^{i} (\delta^{j}(g))^{\binom{i}{j}} \text{ for all } i \in \mathbb{N} \text{ and all } g\in G.$$
\end{lem}
\begin{proof}
We prove this formula by induction on the index $i$.  For $i=1$, there is nothing to prove. Now  suppose that the formula holds for $i=k-1$, where  $k \in \mathbb{N}$, and we prove it for $k$. Then, with the setting $\delta^0(g) = g$ for all $g \in G$, we compute
\begin{eqnarray*}
\phi^{k}(g)&=&\phi( \phi^{k-1}(g))= \phi\bigg(\prod_{j=0}^{k-1} (\delta^{j}(g))^{\binom{k-1}{j}}\bigg)\\ 
&=& \prod_{j=0}^{k-1} (\delta^{j}(g))^{\binom{k-1}{j}}  \delta \bigg(\prod_{j=0}^{k-1} (\delta^{j}(g))^{\binom{k-1}{j}}\bigg)\\
&=& \bigg( \prod_{j=0}^{k-1} (\delta^{j}(g))^{\binom{k-1}{j}} \bigg) \bigg(\prod_{j=0}^{k-1} (\delta^{j+1}(g))^{\binom{k-1}{j}} \bigg)\\
&=& \bigg( \prod_{j=0}^{k-1} (\delta^{j}(g))^{\binom{k}{j} - \binom{k-1}{j-1}} \bigg)  \bigg(\prod_{j=0}^{k-1} (\delta^{j+1}(g))^{\binom{k-1}{j}}\bigg)\\
&=&\bigg( \prod_{j=0}^{k-1} (\delta^{j}(g))^{\binom{k}{j}}  \bigg) \bigg( \prod_{j=1}^{k-1} (\delta^{j}(g)) ^{ - \binom{k-1}{j-1}} \bigg) \bigg( \prod_{j=0}^{k-1} (\delta^{j+1}(g))^{\binom{k-1}{j}} \bigg)\\
&=&\bigg( \prod_{j=0}^{k-1} (\delta^{j}(g))^{\binom{k}{j}}  \bigg)  \bigg( \prod_{j=1}^{k-1} (\delta^{j}(g)) ^{ - \binom{k-1}{j-1}} \bigg) \bigg(\prod_{j=1}^{k} (\delta^{j}(g))^{\binom{k-1}{j-1}}  \bigg)\\
&=& \bigg( \prod_{j=0}^{k-1} (\delta^{j}(g))^{\binom{k}{j}}  \bigg) \delta^{k}(g)=\prod_{j=0}^{k} (\delta^{j}(g))^{\binom{k}{j}}.
\end{eqnarray*}
So the formula holds true for all powers. 
\end{proof}

\begin{rmk} \label{naturalwaylift}
Let $G$ be a finite group, $N$ be a normal subgroup of  $G$. Let  $M$
be a normal abelian subgroup of $G$. If $M \leq Z(N)$, then any derivation $d : G/N \to
M$ extends to a derivation $\delta : G \to M$ defined by $\delta(g)=d(gN)$ for all $g \in G$. It is straightforward to show that $\delta$ is a derivation.
\end{rmk}

The following result appears as Theorem $3.2$ in \cite{caranti:1991} and will  be useful for constructing autormorphims.
\begin{thm}
\label{car}
Let $G=\langle x, y\rangle$ be a metabelian two-generator group. Then the following two conditions are equivalent:
\begin{enumerate}
\item 
For all $a,b\in \gamma_{2}(G)$ there is an automorphism of $G$ that maps $x$ to $xa$ and $y$ to $yb$.
\item
$G$ is nilpotent.
\end{enumerate}
\end{thm}

We now change  the topic and state some useful general results in group theory. For instance, the next result is about the structure of finite $2$-groups of maximal class  \cite[Corollary 3.3.4(iii)]{leedham-green:}.

\begin{lem}
\label{bi}
Let $G$ be a finite $2$-group of maximal class. Then $G$ is isomorphic to  $D_{2^{n}}$, $Q_{2^{n}}$ or $SD_{2^{n+1}}$ for some $n\geq 3$. 
\end{lem}

\begin{rmk} \label{Dih} 
An important connection between Lemma \ref{bi} and the automorphism group of a finite $p$-group is explored  in  \cite[Theorem 1]{shahriari:1987}, where the author also shows  that the generalized quaternion group  and the semidihedral group  cannot be subgroups of a capable group.
\end{rmk}

\noindent
It is also useful to recall a result due to Blackburn  \cite[Theorem 3.2]{blackburn:1958}.

\begin{thm} \label{Black} 
Let $p$ be an odd prime and  $G$ be a finite $p$-group of order $p^n$ such that $G/ \gamma_{2}(G)$ is an elementary abelian group, $\C(G)=m-1$ for some integer $m\geq 1$, and $\gamma_{i}(G) / \gamma_{i+1}(G) \cong C_{p}$ for  $2 \le i \le m-1.$  Then, for $4 \leq m \leq p +1 $, it holds that $G/\gamma_{m-1}(G)$ and $\gamma_{2}(G)$ are groups of exponent $p$.  Moreover, if $m \leq p$ then the elements of order $p$ in $G$ form a characteristic subgroup of index at most $p$ in $G$.  
\end{thm}

To finish with this section, we refer to the classification of  finite $p$-groups of order $p^{4}$, when $p$ is an odd prime greater than $3$,  given by Huppert \cite{huppert:1967} (see Chapter 3, page 346). Analyzing this classification formed by $12$ different isomorphism classes and making calculations, we get the following result.

\begin{thm} \label{goodquotient}
There is only one isomorphism class for finite $p$-groups of maximal class and order $p^{4}$, whenever $p \geq 5$. This isomorphism type, namely the number $12$ in the Huppert's classification, is a semidirect product of an elementary abelian three generator group $\langle b \rangle \times \langle c \rangle \times \langle d \rangle$ with a cyclic group $\langle a \rangle$, satisfying $b^{a}=b$, $c^{a}=bc$, $d^{a}=cd$. This group has the following  presentation: 
$$ \langle a,d \ | \ a^{p}, d^{p}, [d,a,d], [d,a,a,d], [d,a,a,a] \rangle. $$
\end{thm}

\section{Useful tools and results}

This section is mainly about the construction of automorphisms of order $p$ (especially non-inner) of some finite $p$-groups of our interest. First of all,  recall that a finite $p$-group $G$ is said to be \emph{strongly Frattinian} if  $C_{G}(\Phi(G))=Z(\Phi(G))$.

\vspace{4pt}
In view of  Corollary \ref{d2}, in order to prove the existence of a non-inner automorphism of order $p$ in a finite $p$-group $G$, we may assume that the condition  $d(Z_{2}(G)/Z(G)) = d(G)d(Z(G))$ holds true.
This  assumption gives us a reduction, which we rewrite as a hypothesis in the following definition.
\begin{dfn}
We say that a finite $p$-group $G$ satisfies Hypothesis A if  $$d(Z_{2}(G)/Z(G)) = d(G)d(Z(G)).$$
\end{dfn}

From now onwards, we assume that the finite $p$-groups $G$ we are working with have coclass $3$.  

On the other note, the following two remarks related to finite $p$-groups of coclass $3$ satisfying Hypothesis A can be easily proved. 

\begin{rmk}\label{remark1}
Let $G$ be a  finite $p$-group of coclass $3$ (and nilpotency class greater than $3$). Then $G$ satisfies Hypothesis $A$ if and only if $d(G)=2$. Moreover, if one of these  two equivalent conditions holds, then it follows that  $d(Z(G))=1$ and $Z_{2}(G)/Z(G)$ is  isomorphic to an elementary abelian $p$-group of order $p^2$. Thus, if $G$ is a finite $p$-group of coclass $3$ satisfying Hypothesis $A$, then for all $i= 2, ... ,\C(G) -1$, the factors $Z_{i}(G)/Z_{i-1}(G)$ are of exponent $p$. Moreover, $Z_{3}(G)/Z_{2}(G)$ must be a group of order $p$ or $p^{2}.$ 
\end{rmk}

\begin{rmk} \label{z2ab}
Let $G$ be a finite $p$-group of coclass $3$  satisfying Hypothesis $A$. Then it is claimed that $[Z_{2}(G),\Phi(G)]=1$. In fact,  for any $g \in Z_{2}(G)$ we have that $g^{p}$ is central, i.e $[x,g^{p}]=1$ for all $x \in G$. On the other hand, for any $g \in Z_{2}(G)$ and for any $x \in G$, the equality $[x^{p},g]=[x,g]^{p}=[x,g^{p}]=1$ holds. Consequently, $[Z_{2}(G),\Phi(G)]=[Z_{2}(G),G^{p}][Z_{2}(G),\gamma_{2}(G)]=1$. Furthermore, if $G$ is also strongly Frattinian, then the condition $[Z_{2}(G),\Phi(G)]=1$ is equivalent to $Z_{2}(G) \leq Z(\Phi(G))$.
\end{rmk}

The following results describe some  new properties of groups $G$ satisfying Hypothesis $A$. From now onwards we  denote by $c:=c(G)$ the nilpotency class of the group $G$.

\begin{lem} \label{specialheadgroup}
Let $G$ be a finite $p$-group of coclass $3$  satisfying Hypothesis $A$. If $p>2$, then either $G$ is powerful or $G/\gamma_{3}(G)G^{p}$ is an extraspecial group of order $p^{3}$ and exponent $p.$ Indeed, if $p=2$ and $c >3$, then $G$ has a quotient isomorphic to $D_{8}$.
\proof

We first consider the case $p >2$. It is easy to see that $G/\gamma_{3}(G)G^{p}$ is of exponent $p$, that its nilpotency class is at most $2$, and since $d(G)=2$ it follows that this quotient is also a two-generator group. Thus, the  group $G/\gamma_{3}(G)G^{p}$ is either  an elementary abelian group of order $p^{2}$, or an extraspecial group of order $p^{3}$ and exponent $p$. In the former case it follows that $\gamma_{3}(G)G^{p}$ contains $\Phi(G)$. However, this situation only occurs when $[G,G] \leq G^{p}$, i.e  when $G$ is a powerful group. 

\vspace{4pt} Secondly,  consider $p=2$ and let $P:=G/Z_{c-2}(G)$, whose nilpotency class is $2$. We distinguish two possible cases: $|P|=8$ or $\vert P\vert=16$. In the former case, $P$ is of maximal class and also capable, since $G/Z_{c-2}(G) \cong (G/Z_{c-3}(G))/Z(G/Z_{c-3}(G)).$ Thus, from \cite{shahriari:1987} and taking into account Lemma \ref{bi}, $P$ must be a dihedral quotient of $G$, as desired. In the latter case, i.e. $|P|=16$,  we know that $\exp (P/Z(P))=2$. Thus $P^2\subseteq Z(P)$. Indeed, $\Phi(P)=P^2$, and since $d(P)=2$, it follows that  $|P:P^{2}|=4$ and consequently, we get $P^{2}=Z(P)$. If we write $P=\langle x, y\rangle$, then $\gamma_{2}(P)=\langle [x,y]\rangle$, which is contained in $Z(P)=P^2$ and has order $2$. So, there exists $g\in P$ such that $g^2\notin \gamma_{2}(P)$. Take $\overline{P}:=P/\langle g^{2} \rangle $. Then $\langle \bar{g}, \overline{\gamma_{2}(P)} \rangle \leq \Omega_{1}(\overline{P})$ and $\overline{P}$ is a non-abelian group of o
 rder $8$. Consequently, $\overline{P}$ is isomorphic to a dihedral group of order $8$, and thus $G$ has a quotient isomorphic to $D_8$, as desired.
\endproof
\end{lem}

\begin{lem} \label{autu}
Let $G$ be a finite $p$-group of coclass $3$ satisfying Hypothesis $A$, where $p$ is an odd prime, and let $u$ be a non central element in $\Omega_{1}(Z_{2}(G))$. Then $C_{G}(u)$ is maximal in $G$ and the map $\phi_{u}:G \rightarrow G$ defined by $\phi_{u}(x^{i}m)=(xu)^{i}m$, for all $1 \le i \le  p-1$ and for all $m \in C_{G}(u)$, is an automorphism of $G$ of order $p$.
\proof
Let  $u \in \Omega_{1}(Z_{2}(G)) - Z(G)$. First of all, note that  the map $\alpha : G \to \Omega_{1}(Z(G))$ given by $\alpha(x) = [x,u]$ is a homomorphism.  Hence, the subgroup $M:=C_{G}(u)$ is maximal in $G$. Now  take any two generators $x, y$ in $G$ such that $x \in G - M$, $y \in M - \Phi(G)$, and  define the map ${\phi}_u :G \rightarrow G$ such that $\phi_{u}(x^{i}m)=(xu)^{i}m$, for all $1 \le i \le  p-1$ and for all  $m \in C_{G}(u)$, which obviously fixes every element of $M$. Now it is easy to verify that this map ${\phi}_u$  is an automorphism of $G$, and that it has order $p$. 
\endproof
\end{lem}

\begin{lem} \label{gamma3}
Let $G$ be a finite $p$-group of coclass $3$ satisfying Hypothesis $A$,  where $p$ is an odd prime. Assume that  $G$ is strongly Frattinian and that  $\Omega_{1}(Z_{2}(G)) \nleq Z(\gamma_{3}(G)G^{p})$. Then $G$ has a non-inner automorphism of order $p$.
\proof
By  Remark $\ref{z2ab}$ we deduce that $[\Omega_{1}(Z_{2}(G)), \gamma_{3}(G)G^{p}]=1$. Thus, by the hypothesis of the lemma, it follows that   $\Omega_{1}(Z_{2}(G)) \nleq  \gamma_{3}(G)G^{p}$. This implies  that there exists  a non central element $u \in \Omega_{1}(Z_{2}(G))$, which lies in $\Phi(G) - \gamma_{3}(G)G^{p}$. Now  define a map $\phi_{u}$ as in Lemma $\ref{autu}$. If $\phi_{u}$ is an inner automorphism of $G$, then there exists  an element $h \in G -\Phi(G)$ such that $\phi_{u}(g)=g^{h}$, for all $g \in G$. In particular, $h$ centralizes a maximal subgroup of $G$, and so $h$ centralizes $\Phi(G)$. Indeed, since the group $G$ is strongly Frattinian, it follows that $h \in Z(\Phi(G))$, and in particular $h\in \Phi(G)$, which is a contradiction. Hence, $\phi_{u}$ is a non-inner automorphism of $G$ of order $p$, as required.
\endproof
\end{lem}

\begin{lem} \label{YO}
Let $G$ be a finite $p$-group of coclass $3$ satisfying Hypothesis $A$, where $p$ is an odd prime. Assume that  $G$ is strongly Frattinian. Then either $G$ has a non-inner automorphism of order $p$ or $Z_{3}(G)$ contains a subgroup of order at least $p^{4}$ that centralizes $\Phi(G)$. In the latter case, $Z_{3}(G)$ is abelian.
\proof
Since  the group $G$ is strongly Frattinian, we have $Z_{2}(G) \leq Z(\Phi(G))$. Let $u \in  \Omega_{1}(Z_{2}(G)) - Z(G)$, and  $\phi_{u}$ be the induced automorphism of order $p$ as in Lemma $\ref{autu}$. Suppose that  $\phi_u$ is inner.  Then there exists an element $t \in Z_{3}(G) - Z_{2}(G)$ such that $t^{p} \in Z(G)$ and $\phi_{u}(g)=g^{t}$, for all $g \in G$. Set $M:=C_{G}(u)$. Then we have  $t \in C_{G}(M) \leq C_{G}(\Phi(G)) \leq Z(\Phi(G))$. Now  consider the subgroup $H:=\langle Z_{2}(G), t\rangle$ of $Z_3(G)$. Obviously, $H$ is a subgroup of $Z_3(G)$ of order at least $p^{4}$ that centralizes $\Phi(G)$, as required. It is easy to see that with these conditions, $H$ is central and maximal in $Z_{3}(G)$; so $Z_{3}(G)$ is an abelian group.
\endproof
\end{lem}

We finish this section with a remark that will be useful in the next section.

\begin{rmk} \label{tobeinner}
Let $G$ be a finite $p$-group, and  suppose that there exists some $k \in \mathbb{N}$ such that $Z_{k}(G)$ is abelian. Suppose that we can  define, for some $i \in \mathbb{N}$, $p^{i}$ derivations on $G$ that take value in a normal abelian subgroup $H$ contained in $Z_{k}(G)$, and that these derivations extend to $p^{i}$ automorphisms of $G$ of order $p$. Assume that $|Z_{k+1}(G)/Z(G)|=p^{t}$ for some $ t \in \mathbb{N}$. If $t < i$, then we there are at least $p-1$ automorphisms of $G$ that are not inner. On the other hand, if $t=i$, and if we assume  that all these automorphisms are inner, then $[Z_{k+1}(G),G] \leq H$. Furthermore, if every element of $H$ lies in the image of at least one of the derivations produced above, then $[Z_{k+1}(G),G]=H$.
\end{rmk}

\section{Proof of  Main Theorem}

We start reminding the following theorem, which is a reduction to the conjecture.

\begin{thm}\label{known} 
Let $G$ be a finite $p$-group. Then the  conjecture holds true if $G$ satisfies any of the following conditions:

(i) $d(Z_{2}(G)/Z(G)) \ne d(G)d(Z(G))$ (Corollary \ref{d2});

(ii) $G/Z(G)$ is powerful (\cite{abdollahi:2010});

(iii) $G$ is  regular (\cite{schmid:1980});

 (iv) $G$ is not strongly Frattinian (\cite{deaconescu:2002});

(v) $\gamma_2(G)$ is  cyclic (\cite{jamali:2013}).
\end{thm}

\noindent
Let us rewrite these reductions as a hypothesis.

\begin{dfn}
We say that a finite $p$-group $G$ satisfies Hypothesis B, if none of the conditions from (i) to (v) of Theorem \ref{known} holds true for $G$.
\end{dfn}

To begin with, we note that Hypothesis $B$ is much stronger than Hypothesis $A$. Thus all the results proved in the preceding section under Hypothesis A are also true under Hypothesis B, and therefore  can be freely used under  Hypothesis B.  We complete the  proof of  Main Theorem in two parts depending on whether the prime integer $p$ is odd or even.  We first consider the case when $p$ is an odd prime, which constitutes the bulk of the section.

\begin{thm}\label{marco1}
Let $G$ be a finite $p$-group of coclass $3$ satisfying Hypothesis $B$, where $p$ is an odd prime. Then $G$ admits a non-inner automorphism of order $p$, if any one of the following conditions holds true:

(i)  $Z(G)$ is cyclic of order $p^{2}$;

(ii) $|Z(G)| = p$ and  $[Z_{3}(G),\Phi(G)]=1$;

(iii)  $|Z(G)| = p$,   $[Z_{3}(G),\Phi(G)] \neq 1$ and $Z_{3}(G)/Z(G)$ is not elementary abelian;

(iv)  $|Z(G)| =p$,   $[Z_{3}(G),\Phi(G)] \neq 1$ and both  $Z_{3}(G)/Z(G)$ as well as  $Z_{2}(G)$ are elementary abelian. 
\proof
Throughout the proof of this theorem $c>3$  denotes the nilpotency class of $G$. Since the group $G$ satisfies Hypothesis $B$, by Remark \ref{remark1} we have $d(G)=2$, $Z_{2}(G)/Z(G)$ is isomorphic to an elementary abelian group of order $p^2$ and $|Z_{3}(G)/Z_{2}(G)| \in \{p,p^{2}\}$, and moreover, by Lemma \ref{specialheadgroup} we can assume that $G/\gamma_{3}(G)G^{p}$ is an extraspecial group of order $p^{3}$ and exponent $p$ (otherwise the group $G$ would be powerful, and this contradicts Hypothesis $B$).  Indeed, by Lemma \ref{gamma3} we can also assume that $\Omega_{1}(Z_{2}(G)) \leq Z(\gamma_{3}(G)G^{p})$, and by Lemma \ref{YO}, that there exists a subgroup of order at least $p^{4}$ in $Z_{3}(G)$, which centralizes $\Phi(G)$.

\vspace{4pt}

On the other note, let $F$ be the free group generated by two elements, say, $x$ and $y$. We know that $G$ and $G/\gamma_{3}(G)G^{p}$ are both suitable quotients of $F$. Since, clearly, $Z_{2}(G)$ is an abelian normal subgroup of $G$, at the same time, we can  view $Z_{2}(G)$ as an $F$-module, a $G$-module and a $G/\gamma_{3}(G)G^{p}$-module.  Define a family $\Delta$ consisting of derivations of the type $\delta_{g_{1},g_{2}} : F \rightarrow  \Omega_{1}(Z_{2}(G))$ such that $\delta_{g_{1},g_{2}}(x)= g_{1}$ and $\delta_{g_{1},g_{2}}(y)=g_{2}$, where  $g_{1},g_{2} \in \Omega_{1}(Z_{2}(G))$. According to Lemma \ref{determineunique}, these maps are uniquely determined by assigning the value on $x$ and  $y$. 

\vspace{4pt}

Next we claim that all the derivations in  $\Delta$ preserve the  defining relations of the quotient group $G/\gamma_{3}(G)G^{p}$. Indeed, we can assume that $G/ \gamma_{3}(G)G^{p}$ admits the presentation  $ \langle x,y  \ | \ x^{p}, y^{p}, [y,x,y], [y,x,x] \rangle$. Thus, in other words, we have to prove that for any $\delta \in \Delta$,  the relations $\delta(x^{p})=1$, $\delta(y^{p})=1$, $\delta([y,x,y])=1$, and $\delta([y,x,x])=1$ hold true. To start with it, in fact, since $\delta(x) \in \Omega_{1}(Z_{2}(G))$ and $p > 2$,  we get 
\[\delta(x^{p})=\delta(x)^{x^{p-1}}\delta(x^{p-1})= \cdots = \delta(x)^{x^{p-1}+ \cdots +1}=\delta(x)^{p}[\delta(x),x]^{\binom{p}{2}}=1,\]
 for all $\delta \in \Delta$. For all $\delta \in \Delta$, the relation $\delta(y^{p})=1$ follows on the similar lines. (We remark that these  two relations  hold true only when $p>2$, otherwise, the commutator $[\delta(x),x]^{\binom{p}{2}} = [\delta(x),x]$ does not vanish.)

\vspace{4pt}

For any $\delta \in \Delta$  and $g,h \in G$, we have $\delta(gh) = \delta(hg)^{[g,h]}\delta([g,h])$, and since $[Z_{2}(G), \gamma_{2}(G)]=1$, it follows that $\delta(gh) = \delta(hg)\delta([g,h])$. Thus $\delta(g)^{h}\delta(h)= \delta(h)^{g} \delta(g) \delta([g,h])$. Now, since  $Z_{2}(G)$ is a normal abelian subgroup of $G$, it follows that 
\begin{eqnarray*}
\delta([g,h]) &=& \delta(g)^{-1} (\delta(h)^{g})^{-1} \delta(g)^{h}\delta(h) = \delta(g)^{-1} \delta(g)^{h} (g^{-1} \delta(h)g)^{-1} \delta(h)\\
& = & [\delta (g),h][g, \delta(h)],
\end{eqnarray*}
 which is an element of $Z(G)$.

\vspace{4pt}

Let $w:=[y,x]$. Now putting $g = w$ and $h = y$ in the preceding equation, and noting that $\delta(w) \in Z(G)$ and that  $\delta(y) \in Z_2(G)$, we get 
\[\delta([y,x,y]) = \delta([w,y]) = [\delta(w),y][w,\delta(y)]=1.\]
 Similarly, we can prove that $\delta([y,x,x]) =1$. Thus, these equalities settle  the claim about the preservation of the defining relations of  $G/\gamma_{3}(G)G^{p}$.

\vspace{4pt}
By applying Lemma \ref{Der} we obtain $|\Omega_{1}(Z_{2}(G))|^{2}$ derivations from the group $G/\gamma_{3}(G)G^{p}$ to $\Omega_{1}(Z_{2}(G))$. Indeed, since  $\Omega_{1}(Z_{2}(G)) \leq Z(\gamma_{3}(G)G^{p})$, using Remark \ref{naturalwaylift} it is possible to lift these derivations to $|\Omega_{1}(Z_{2}(G))|^{2}$ derivations from $G$ to $\Omega_{1}(Z_{2}(G))$, in a natural way. Moreover, these derivations are trivial on $\Omega_{1}(Z_{2}(G)) $ and thus, by Lemma \ref{lift} these derivations can be extended to $|\Omega_{1}(Z_{2}(G))|^{2}$ automorphisms of $G$. 

\vspace{4pt}
Next we claim that the  automorphisms obtained in the preceding paragraph are of order $p$. Let $\phi$ be such an  automorphism. Then there exists a derivation $\delta$ from $G/\gamma_{3}(G)G^{p}$ to $\Omega_{1}(Z_{2}(G))$ such that  $\phi(g)=g\delta(g\gamma_{3}(G)G^{p})$, for all $g\in G$. Let us start calculating $\phi^2 := \phi \circ \phi$. In fact, for all $g\in G$ 
\begin{eqnarray*}
\phi^{2}(g) &=&\phi(g\delta(g\gamma_{3}(G)G^{p}))=\phi(g)\phi(\delta(g\gamma_{3}(G)G^{p}))\\
&=& g(\delta(g\gamma_{3}(G)G^{p}))^{2}.
\end{eqnarray*}
By a simple inductive argument, we prove that for all $g\in G$, $\phi^{p}(g)=g(\delta(g\gamma_{3}(G)G^{p}))^{p}=g$, and this proves our claim.

\vspace{4pt}

First of all,  suppose that $Z(G) \cong C_{p^{2}}$. Then, as above, we obtain at least $p^{4}$ automorphisms induced by derivations that take value in $\Omega_{1}(Z_{2}(G))$. If all of them are inner, then each of them is induced by conjugation by some element of $Z_{3}(G)$ modulo  $Z(G)$. However, since $|Z_{3}(G)/Z(G)|=p^{3}$, this is not possible. 
So we can assume, for the rest of the proof, that $Z(G) \cong C_{p}.$ We now suppose that  $[Z_{3}(G),\Phi(G)]=1$. In this case,  fix a non central element $u \in \Omega_{1}(Z_{2}(G))$. By Lemma \ref{autu} we know that  $C_{G}(u)$ is a maximal subgroup of $G$. Choose two generators $x,y$  of $G$ such that $[u,y]=1$ and $[x,u]\neq 1$. Let $\delta := \delta_{1,u}$ be a derivation from $\Delta$ whose values on the generating elements are given by  $\delta(x) = 1$ and $\delta(y) = u$.  Since $\delta ([x,y]) = [x,u] \neq 1 $, the derivation $\delta$ induces an automorphism of $G$ of order $p$, say, $\phi$, that does not centralize $\Phi(G)$. If  $\phi$ were inner, then it would be induced by conjugation of an element $t\in Z_{3}(G) - Z_{2}(G)$ such that $[t,\Phi(G)] \neq 1$. However, the last statement is a contradiction with  the assumption $[Z_{3}(G),\Phi(G)] = 1$. Hence, the automorphism $\phi$  must be non-inner, and we get the result of the theorem, in this 
case. 

\vspace{4pt}

Now assume that $[Z_{3}(G),\Phi(G)] \neq 1$. Since, by Lemma $\ref{YO}$, we are assuming that there exists a subgroup in $Z_3(G)$ of order exactly $p^4$ which centralizes $\Phi (G)$, we can also assume, without any loss of generality, that $|Z_{3}(G)|$ is exactly equal to $p^{5}$.  Indeed, we can as well produce at least $p^4$ derivations by assigning values on generators of $G$  in the second center, and consequently,  we obtain  at least $p^{4}$ automorphisms of $G$ of order $p$. Furthermore, in the particular case when $|\Omega_{1}(Z_{3}(G)/Z(G))| \leq p^{3}$, i.e. when $Z_{3}(G)/Z(G)$ is not elementary abelian, there are at most $p^{3}$ inner automorphims of $G$ of order $p$ induced by elements of $Z_{3}(G)$. Thus by a simple argument of counting,  the statement of the theorem holds true, as well, in this case. Finally, if we assume that $|\Omega_{1}(Z_{3}(G)/Z(G))|= p^{4}$ (i.e $Z_{3}(G)/Z(G)$ is elementary abelian) and that $Z_{2}(G)$ is an elementary abelian group, the
 n the theorem again holds true by Remark \ref{tobeinner}, since in this  case the total number of automorphims of $G$ of order $p$ is equal to $|\Omega_{1}(Z_{2}(G))|^{2}=|Z_{2}(G)|^{2} =p^{6}$. 
\endproof
\end{thm}

From now onwards,  in order not to repeat the cases for which the conjecture holds true,  we set a new  hypothesis.  Although in the remaining study we  deal with finite $p$-groups, where $p \ge 5$, the following definition of Hypothesis $C$ is valid for every odd prime.

\begin{dfn}
For an odd prime $p$, we say that a finite $p$-group $G$ satisfies Hypothesis $C$ if $G$ respects Hypothesis $B$ and  all of the following properties hold true:

 (1)  $|Z(G)| = p$;

(2) $[Z_{3}(G),\Phi(G)] \neq1$;

(3) $Z_{3}(G)/Z(G)$ is elementary abelian;

(4) $Z_{2}(G)$ is not elementary abelian.
\end{dfn}

Before we weave the next thread in the proof of the Main Theorem, we review  some of the useful properties we have obtained so far for a finite $p$-group $G$ of nilpotency class $c>3$ and  coclass $3$, satisfying Hypothesis $C$. First of all, we recall that all the quotients of the upper central series of $G$ are elementary abelian. Furthermore, $\Phi(G)=Z_{c-1}(G)$, $G/\Phi(G) \cong C_{p} \times C_{p}$, $Z(G)\cong C_p$, $Z_2(G)/Z(G)\cong C_p\times C_p$, $\vert Z_3(G)\vert=p^5$,  $Z_{3}(G)/Z_{2}(G)$ is elementary abelian isomorphic to $ C_{p} \times C_{p}$, and all the other quotients of the upper central series of $G$ are isomorphic to $C_{p}$.
Moreover, $G/Z_{3}(G)$ is of maximal class, $Z_{2}(G) \cong C_{p^{2}} \times C_{p}$, $\gamma_{3}(G)G^{p}=Z_{c-2}(G)$, and $Z_{3}(G)$ is an abelian group which does not centralize the Frattini subgroup of $G$.

Finally, if $\vert G\vert=p^n$, since $c\geq 4$ and $cc(G)=3$,  then clearly $n\geq 7$. In the case $n=7$ with $p\geq 5$, since $c=7-cc(G)=7-3=4\leq p-1$, it follows from Theorem $10.2$ of \cite{huppert:1967} that $G$ must be regular.  Since we are dealing with irregular $p$-groups, we can  assume that $n\geq 8$.



\vspace{4pt}

In view of the preceding discussion,  we now continue with the proof of the Main Theorem.

\begin{lem}
\label{five}
Let $p \ge 5$ be an odd prime  and  $G$ be a finite $p$-group of coclass $3$ which respects Hypothesis $C$. If the nilpotency class of $G$ is  $5$, then $G$ has a non-inner automorphism of order $p$. 
\end{lem}
\begin{proof}
We remark that the results developed through the following four paragraphs do not require the condition that the nilpotency class of $G$ is $5$. We'll also use these results in the proof of Theorem \ref{main}.

\vspace{4pt}

 Let  $u \in \Omega_{1}(Z_{2}(G)) - Z(G)$. According to the structures of $Z(G)$ and $Z_2(G)$, we can write $Z_2(G) = \langle v\rangle\times \langle u\rangle$, where $v \in Z_2(G)$ such that $v^p$ generates $Z(G)$. Let $z$ generate $Z(G)$. Then  $Z(G)=\langle v^p \rangle=\langle z\rangle$. Since $|Z(G)| = p$, it follows that $M:=C_{G}(u)$ is maximal in $G$. Choosing  generators $x, y$ for $G$ such that $x \in G - M$, $y \in M - \Phi(G)$, we define $\phi_{u}$ as in Lemma \ref{autu}. Without any loss of generality, we can also assume that $[x,u]=z$, $[x,v]=z$ and $[y,v]=z$. 

\vspace{4pt}

If  $\phi_{u}$ is a non-inner automorphism of order $p$, then we are done. Otherwise, there exists an element $t \in Z_{3}(G) - Z_{2}(G)$ such that $[x,t]=u, [y,t]=1$, $t^{p} \in Z(G)$ and $t \in Z(M)$. In particular, since $Z(M) \leq C_{G}(\Phi(G))= Z(\Phi(G))$, it follows that $t\in Z(\Phi(G))$ as well. Now  consider the subgroup $\langle Z_{2}(G),t \rangle$. This subgroup is abelian of exponent $p^{2}$, and taking into account that $t^{p} \in Z(G)$, this subgroup cannot be a two-generator subgroup. Since $\langle Z_{2}(G),t \rangle$ is abelian, we can always find an element $h \in  \langle Z_{2}(G), t \rangle - Z_2(G)$ such that 
\[\langle Z_{2}(G),t \rangle  \cong \langle u \rangle \times \langle v^p \rangle \times \langle h \rangle. \]
Note that $h \in Z(\Phi(G))$.  In particular, we have that $t=v^{i}u^{j}h^{k}$, for some integers $i, j, k$.  Since $[x,t]=u$,  we get $u=[x,v^{i}u^{j}h^{k}]=[x,v^{i}][x,u^{j}][x,h^{k}]=z^{i+j}[x,h^{k}]$, and consequently $[x,h^{k}]=z^{-i-j}u$, which is an element of  $\Omega_{1}(Z_{2}(G))$. Similarly, since $[y,t]=1$, we get 
\[1=[y,v^{i}u^{j}h^{k}]=[y,v^{i}][y,u^{j}][y,h^{k}]=[y,v^{i}][y,h^{k}]=z^{i}[y,h^{k}],\]
 and consequently $[y,h^{k}]=z^{-i}$, which is an element of $Z(G)$.  Set $w:=h^{k}$.  Then we can write  
\[\Omega_{1}(\langle Z_{2}(G),t \rangle )= \langle u \rangle \times \langle v^p \rangle \times \langle w \rangle,\]
\[ \Omega_{1}(\langle Z_{2}(G))=\langle u \rangle \times \langle v^p \rangle=\langle u \rangle \times \langle z \rangle\]
 and 
\[Z(G)=\langle v^p\rangle=\langle z\rangle.\]

\vspace{4pt}

We now define a family of assignments $ x \rightarrow a$, $y \rightarrow b$, with $a \in \Omega_{1}( \langle Z_{2}(G), t \rangle)$  and $b \in \Omega_{1}( Z_{2}(G))$. Next we check that these assignments extend to derivations that preserve the relations of $G/Z_{c-2}(G)=G/{\gamma}_3(G)G^p$. Recall that, in this case, $G/{\gamma}_3(G)G^p$ is an extraspecial $p$-group of exponent $p$ and order $p^3$, and have presentation 
\[\langle x,y \ | \ x^{p}, \ y^{p}, \ [y,x,x], \ [y,x,y]  \rangle.\]

\vspace{4pt}

Let $\delta$ be any derivation obtained by one of the above assignments. Since $p \ge 5$, we have 
 \[\delta(x^{p})=\delta(x)^{p}[\delta(x),x,x]^{\binom{p}{3}} =1\]  
and  
\[\delta(y^{p})=\delta(y)^{p}[\delta(y),y,y]^{\binom{p}{3}}=1.\]
 Further, since  $\Omega_{1}( \langle Z_{2}(G), t \rangle) \leq C_{G}(M) \leq Z(\Phi(G))$,  applying $\delta$ to the equation $yx=xy[y,x]$,  we get  $\delta(y)^{x}\delta(x)=\delta(x)\delta(y)\delta([y,x])$, and consequently $\delta([y,x])=[\delta(y),x]$ is an element of $Z(G)$. 

\vspace{4pt}

Now using  the fact that $\delta([y,x])\in Z(G)$, applying  $\delta$ to the equality $[y,x]x = x[y,x][y,x,x]$, it follows that $\delta([y,x])\delta(x)=\delta(x)\delta([y,x])\delta([y,x,x])$, and as a consequence, $\delta([y,x,x])=1$. Analogously, using the equality $[y,x]y=y[y,x][y,x,y]$ and proceeding as above, we get  $\delta([y,x])\delta(y)=\delta(y)\delta([y,x])\delta([y,x,y])$, which implies that $\delta([y,x,y])=1$. Since 
\[\Omega_{1}(\langle Z_{2}(G),t\rangle ) \leq Z_3(G) \le Z(Z_{c-2}(G)) = Z(\gamma_{3}(G)G^{p}),\]
by Remark \ref{naturalwaylift} we obtain $p^5$ derivations of $G$ taking values in $\Omega_{1}(\langle Z_{2}(G),t\rangle )$. These derivations then give rise to $p^5$ automorphisms of $G$ of order $p$. If some of the automorphisms obtained above is non-inner, then we are done. Otherwise, since $|Z_4(G)/Z(G)| = p^5$, it follows from Remark  \ref{tobeinner} that  $[Z_{4}(G),G] = \Omega_{1}(\langle Z_{2}(G),t\rangle )$. Thus  every element of $Z_4(G)$ induces an automorphism of order $p$ by conjugation, and therefore it follows that  $Z_{4}(G)/Z(G)$ is elementary abelian. This implies that $Z_{4}(G)$ is a five generator group of exponent $p^{2}$ and of the nilpotency class  $\le 2$. 

Recall that, since the nilpotency class of $G$ is $5$,  $Z_{3}(G)=Z_{c-2}(G)=\gamma_{3}(G)G^{p}$ and $Z_{4}(G)=\Phi(G)$. Continuing with the preceding paragraph,  if the nilpotency class of $Z_{4}(G)$ is $1$, it follows that $\Phi(G) = Z_4(G)$ is an abelian group,  and consequently $[Z_{3}(G),\Phi(G)]=1$, which is in contradiction with the given hypothesis. So assume that  the nilpotency class of $Z_{4}(G)$ is  $2$. In this  case we explore  the lower central series of $G$. Clearly $\gamma_{5}(G)=Z(G)$. Now $p^{2} \leq |\gamma_{4}(G)|\leq p^{3}$. Indeed, 
\[|\gamma_{4}(G)| = |[\gamma_{3}(G),G]| \leq |[Z_{3}(G),G]|=|\Omega_{1}(Z_{2}(G))|=p^{2},\]
 and consequently $\gamma_{4}(G)= \Omega_{1}(Z_{2}(G))$. Now, since 
\[p^3\leq |\gamma_{3}(G)| = |[\gamma_{2}(G),G]| \leq |[Z_{4}(G),G]|=|\Omega_{1}(\langle Z_{2}(G),t \rangle)|=p^{3},\]
 we have $\gamma_{3}(G)=\Omega_{1}(\langle Z_{2}(G),t \rangle)$. Thus, since the group $G$  is generated by two elements, the quotient $\gamma_{2}(G)/\gamma_{3}(G)$ is cyclic of order $p$ or $p^{2}$. If  $\gamma_{2}(G)/\gamma_{3}(G) \cong C_{p^{2}}$, then $G/\gamma_{2}(G) \cong C_{p^{2}} \times C_{p}$, which is not possible. As a consequence, $\gamma_{2}(G)/\gamma_{3}(G)$ is cyclic of order $p$. Since $\gamma_{3}(G)$ is maximal central subgroup in $\gamma_{2}(G)$,  this implies that $\gamma_{2}(G)$ is abelian. Thus $G$ is a two generator metabelian $p$-group. By Theorem \ref{car},  it now follows that every assignment $x \rightarrow a$ and $y \rightarrow b$ with $a,b \in \gamma_{3}(G)$ extends to an automorphism of $G$. We claim that  all of these assignments  extend to automorphisms of order $p$. 
Let $\delta$ be a derivation associated to such an assignment, and let $\phi$ be the induced automorphism of $G$, i.e. $\phi(g) :=g\delta(g)$, for all $g \in G$. Then by Lemma \ref{powers} we have for all $i \in \mathbb{N}$ and all $g\in G$ that
$$\phi^{i}(g)=\prod_{j=0}^{i} (\delta^{j}(g))^{\binom{i}{j}}.$$

It follows that all the terms in the above product except for the first three terms are $1$, i.e. 
\[\phi^{i}(g)=g\delta(g)^{i}(\delta(\delta (g)))^{\binom{i}{2}}.\]
In fact, by Lemma \ref{free}, $\delta(G) =\gamma _{3}(G)$, $\delta^{2}(G)= \delta(\gamma_{3}(G))=\gamma_{5}(G)=Z(G)$ and $\delta^{3}(G)=\{1\}$. Therefore, since $p \geq 5$ and the nilpotency class of the group is $5$, all of these  automorphisms have order $p$. Moreover, since $|\gamma_{3}(G)|^{2}=p^{6}$ and $\gamma_{3}(G) \leq \Omega_{1}(Z_{3}(G)) \cap \Omega_{1}(\gamma_{2}(G))$, we obtain $p^{6}$ automorphisms of $G$ of order $p$ induced by the above process.  Applying Remark \ref{tobeinner} once more, the proof of the theorem is now complete. 
\end{proof}

\vspace{4pt}

Notice that, in the preceding Lemma, if $p \geq 7$, then the group $G$ becomes regular. So the preceding  lemma is necessary,  only when $p=5$. In the light of this lemma, we now onwards assume that the nilpotency class $c$ of the group we are considering is at least $6$. We now prepare to finishing the proof of  Main Theorem for all $p \ge 5$.  Before  weaving the final thread, we analyze the quotient  $G/Z_{c-3}(G)$ for any finite $p$-group $G$ of coclass $3$ satisfying Hypothesis $C$ in the following result.

\begin{lem}\label{expogood}
Let  $p \geq 5$ and  $G$ be a finite $p$-group of coclass $3$  and nilpotency class at least $6$ such that $G$ satisfies Hypothesis $C$.  Then $G/Z_{c-3}(G)$ is a maximal class group of exponent $p$.
\end{lem}
\begin{proof}
First of all, with the conditions of the lemma, it is easy to prove that   $G/Z_{c-3}(G)$ is a group of order $p^4$ and nilpotency class $3$. In particular, $G/Z_{c-3}(G)$ is a maximal class group. It only remains to prove that  the exponent of this quotient is $p$. To prove this,  we consider the quotient $G/Z_{c-4}(G)$, and  split the proof in two cases: $c\geq 7$ and $c=6$. In the former case,  the quotient $G/Z_{c-4}(G)$ is of maximal class, as $c-4 \ge 3$. Applying Theorem \ref{Black} to the group $G/Z_{c-4}(G)$, we get that $G/Z_{c-4}(G)/Z(G/Z_{c-4}(G))$, which is isomorphic to $G/Z_{c-3}(G)$,  is in fact, of exponent $p$. In the latter case, i.e. $c=6$, clearly $Z_{c-3}(G)=Z_{3}(G)$. Set $\overline{G}=G/Z_{2}(G)$, and  consider $\gamma \in Z_{2}(\overline{G})-Z(\overline{G})$. Since $\gamma \not\in Z(\overline{G})$, there must exists an element $g$ (say) in the generating set of $G$ such that $[\gamma, g] \neq 1$. Since $[\gamma, g] \in Z(\overline{G})$, there exists a
 n element $h \in Z(\overline{G})$ such that $Z(\overline{G})  = \langle [\gamma, g], h \rangle$. Clearly, the cyclic subgroup $\langle h \rangle$ of order $p$  is central, and therefore normal in $\overline{G}$. Finally, factoring out the group $\overline{G}$ by  $\langle h \rangle$, we obtain a maximal class group of class $4$. That the exponent of $G/Z_{c-3}(G)$ is $p$, now follows from Theorem \ref{Black} as argued above. 
\end{proof}

Before proceeding further, we recall some basic definitions on finite $p$-groups of maximal class. Let $G$ be a $p$-group of maximal class of order $p^n$. We write $G_i={\gamma}_i(G)$ for $i\geq 2$ and $G_0=G$. We define $G_1=C_G(G_2/G_4)$ (the action of $G$ on $G_2/G_4$ being induced by conjugation). We consider the so-called two steps centralizers $C_G(G_i/G_{i+2})$ for $1\leq i\leq n-2$. All these subgroups are characteristic and maximal in $G$. Since $[G_1, G_1]=[G_1, G_2]\leq G_4$, we have that $C_G(G_1/G_3)=G_1$ and, consequently, it is enough to consider the two-step centralizers for $2\leq i\leq n-2$. We say that $s\in G$ is a uniform element if $s\notin {\bigcup}_{i=2}^{n-2}C_G(G_i/G_{i+2})$. The reader is referred to \cite{leedham-green:} for more details on this material.

\vspace{4pt}

The following result completes the proof of Main Theorem for $p \ge 5$.

\begin{thm}
\label{main}
Let $p \ge 5$, and  $G$ be a finite $p$-group of coclass $3$ satisfying Hypothesis $C$, whose nilpotency class is at least $6$. Then $G$ has a non-inner automorphism of order $p$.
\end{thm}
\begin{proof}
Let $u \in  \Omega_{1}(Z_{2}(G)) - Z(G)$. Then  $M:=C_{G}(u)$ is maximal in $G$, and we can choose an element $y \in   M - \Phi(G)$ as one of the two generators of $G$. Set  $\overline{G}=G/Z_{c-3}(G)$. This quotient has only one two-step centralizer, say $\overline{P}$. Set $\overline{M}:=M/Z_{c-3}(G)$. Thus, either $\overline{M} = \overline{P}$ or $\overline{M} \neq \overline{P}$. We split the proof into two cases, namely: Case 1.  $\overline{M} = \overline{P}$; Case 2. $\overline{M} \neq \overline{P}$.

\vspace{4pt}

\noindent {\bf Case 1.} 
In this case  the image $\overline{y} \in G/Z_{c-3}(G)$ is not a uniform element. So we can choose the other generator $x$ (say) for $G$ such that $\overline{x}$ is uniform. Now  combining Lemma \ref{expogood} and Theorem \ref{goodquotient}, it follows  that $G/Z_{c-3}(G)$ has the following presentation,
$$G/Z_{c-3}(G) = \langle x,y \ | \ x^{p}, y^{p} , [y,x,y], [y,x,x,y], [y,x,x,x] \rangle. $$ 

\vspace{4pt}

Define, as in Lemma \ref{autu}, an automorphism of $G$ of order $p$ that sends $x$ to $xu$ and centralizes $M$. If this automorphism is of order $p$, then we are done. Otherwise, as in the first half of the proof of Lemma \ref{five}, we can write
\[\Omega_{1}(\langle Z_{2}(G),t \rangle )= \langle u \rangle \times \langle v^p \rangle \times \langle w \rangle,\]
\[ \Omega_{1}(\langle Z_{2}(G))=\langle u \rangle \times \langle v^p \rangle=\langle u \rangle \times \langle z \rangle.\]

\vspace{4pt}

Next we consider all possible  assignments from the two generator free group to $\Omega_{1}(\langle Z_{2}(G),t \rangle ) = \langle u \rangle \times \langle z \rangle \times \langle w \rangle$, which then extend in a unique way to a family of $p^{6}$ derivations, from the two generator free group to  $\Omega_{1}(\langle Z_{2}(G),t \rangle )$.  We now prove that  these derivations preserve the relations defining $G/Z_{c-3}(G)$.  In fact,  let  $\delta$ be an arbitrary  derivation obtained  above. Then, since $p \geq 5$, we have
$$\delta(x^{p})=\delta(x)^{p}[\delta(x),x]^{\binom{p}{2}}[\delta(x),x,x]^{\binom{p}{3}}=1$$
and
$$\delta(y^{p})=\delta(y)^{p}[\delta(y),y]^{\binom{p}{2}}[\delta(y),y,y]^{\binom{p}{3}}=1. $$

\vspace{4pt}

Since $\Omega_{1}(\langle Z_{2}(G),t \rangle ) \leq Z(\Phi(G))$,  we get $\delta([y,x])=[\delta(y),x][y,\delta(x)]$, which is an element of $Z(G)$ (and in particular, an element of  $\Omega_{1}(Z_{2}(G))$). Moreover, since $[y,x]y=y[y,x][y,x,y]$, we get 
$$\delta([y,x])^{y}\delta(y)=\delta(y)^{[y,x][y,x,y]}\delta([y,x])^{[y,x,y]}\delta([y,x,y]).$$
Now taking into account that $\Omega_{1}(Z_{2}(G)) \leq C_{G}(y)$ and that $\Omega_{1}(\langle Z_{2}(G),t \rangle ) \leq Z(\Phi(G))$, we get 
$\delta([y,x])\delta(y)=\delta(y)\delta([y,x])\delta([y,x,y])$, and consequently $\delta([y,x,y])=1$. Similarly, we can prove that the equalities $\delta([y,x,x,y])$ $=1$ and $\delta([y,x,x,x])=1$ hold.

\vspace{4pt}

Thus we obtain $p^{6}$ derivations from $G/Z_{c-3}(G)$ to $\Omega_{1}(\langle Z_{2}(G),t \rangle )$. Moreover, $\Omega_{1}(\langle Z_{2}(G),t \rangle ) \leq Z(Z_{c-3}(G))$ since $\Omega_{1}(\langle Z_{2}(G),t \rangle ) \leq Z(\Phi(G))$ and $\Phi(G)=Z_{c-1}(G)$. As above these derivations now extend to $p^{6}$ automorphisms of $G$ of order $p$. Hence, since $|Z_{4}(G)/Z(G)|=p^{5}$, the statement follows by Remark \ref{tobeinner}.

\vspace{4pt}

\noindent {\bf Case 2.} In this case we can choose an element $x$ as the second generator for $G$ such that $\overline{x}$ is not uniform in $\overline{G}$, and therefore we have 
$$G/Z_{c-3}(G)=\langle x,y \ | \ x^{p}, y^{p} , [x,y,x], [x,y,y,x], [x,y,y,y] \rangle .$$

 By the first half of the proof of Lemma \ref{five}, either $G$ admits a non-inner automorphism of order $p$ or   $[Z_{4}(G),G]=\Omega_{1}(\langle Z_{2}(G),t\rangle )$, where $t \in Z_3(G)$ and centralizes $\Phi(G)$. 
Now, using some different type of derivations, we show that either $G$ admits a non-inner automorphism of order $p$ or  $[Z_{4}(G),G]=\Omega_{1}(\langle Z_{2}(G), k\rangle )$, where $k \in Z_3(G)$ which does not centralize $\Phi(G)$.  Consider the assignment $ x \rightarrow 1$ and $y \rightarrow u$. As explained in the proof of Theorem \ref{marco1},  this  assignment extends to an automorphism of $G$ of order $p$,  which does not centralize $\Phi(G)$. If this is a non-inner automorphism of $G$, then we are done. Otherwise, there exists an element $k \in Z_{3}(G) - Z_{2}(G)$ such that $[x,k]=1$, $[y,k]=u$ and $k^{p} \in Z(G)$. Now consider the abelian subgroup $\Omega_{1}( \langle Z_{2}(G), k \rangle)$.  Then we can find and element $s \in  \langle Z_{2}(G), k \rangle$ of order $p$ such that   $[x,s] \in Z(G)$, $[y,s] \in \Omega_{1}(Z_{2}(G))$ and 
\[ \Omega_{1}( \langle Z_{2}(G), k \rangle) =  \langle z \rangle \times \langle u \rangle \times \langle s \rangle.\]

Next we define the family of assignments $ x \rightarrow a$, $y \rightarrow b$ with $a \in \Omega_{1}(Z_{2}(G))  $ and $b \in \Omega_{1}( \langle Z_{2}(G),k \rangle)$. We now prove that these assignments extend to derivations that preserve the relations of $G/Z_{c-3}(G)$.  Let $\delta$ be any derivation obtained by one of the assignments above. Then, as in Case 1, we can easily show that $\delta(x^{p})=1$ and $\delta(y^{p})=1$.

\vspace{4pt}

Now we consider the commutator relations. Since $xy=yx[x,y]$,  we have $$\delta(x)^{y}\delta(y)=\delta(y)^{x[x,y]}\delta(x)^{[x,y]}\delta([x,y]).$$  
Thus
\begin{eqnarray*}
\delta([x,y]) &=&(\delta(y)^{-1})^{x[x,y]}(\delta(x)^{-1})^{[x,y]}\delta(x)^{y}\delta(y)\\
 &=& (\delta(y)^{-1})^{x[x,y]}\delta(y)  (\delta(x)^{-1})^{[x,y]} \delta(x) \delta(x)^{-1}  \delta(x)^{y}\\
&=& [x[x,y],\delta(y)][x,y,\delta(x)][\delta(x),y]\\
&=& [x,\delta(y)]^{[x,y]}[x,y,\delta(y)][x,y,\delta(x)][\delta(x),y]\\
&=& [x,\delta(y)][x,y,\delta(y)][x,y,\delta(x)][\delta(x),y]\\
&=& [x,\delta(y)][\delta(x),y][x,y,\delta(y)][x,y,\delta(x)].
\end{eqnarray*}
Since $\delta(x) \in \Omega_{1}(Z_{2}(G))$, we obtain $\delta([x,y])=[x,y,\delta(y)][x,\delta(y)]$.  Moreover, since $[Z_{3}(G), \gamma_{2}(G)] \leq Z(G)$ and $[x,\Omega_{1}( \langle Z_{2}(G),k \rangle)  ] \leq Z(G)$, it follows that $\delta([x,y]) \in Z(G)$.

Applying $\delta$ to the equation $[x,y]x=x[x,y][x,y,x]$, we get $\delta([x,y])\delta(x)=\delta(x)\delta([x,y])\delta([x,y,x])$, which implies that  $\delta([x,y,x])=1$. Similarly   applying $\delta$ to $[x,y]y=y[x,y][x,y,y]$ gives $\delta([x,y])\delta(y)=\delta(y)^{[x,y]}\delta([x,y])\delta([x,y,y])$, which then implies that $\delta([x,y,y])=[x,y,\delta(y)]$ is an element of $Z(G)$. The next step is to consider the commutators of weight $4$. From $[x,y,y]x=x[x,y,y][x,y,y,x]$, we get $\delta([x,y,y])\delta(x)=\delta(x)\delta([x,y,y])\delta([x,y,y,x])$, and therefore $\delta([x,y,y,x])=1$. Similarly, from $[x,y,y]y=y[x,y,y][x,y,y,y]$, the equality $\delta([x,y,y])\delta(y)$ $=\delta(y)\delta([x,y,y])\delta([x,y,y,y])$ holds, and consequently $\delta([x,y,y,y])=1$.

\vspace{4pt}

Thus, we obtain $p^{5}$ derivations that take value in $\Omega_{1}(\langle Z_{2}(G),k\rangle )$, and  extend to automorphisms of $G$ of order $p$.  If one of these automorphisms is non-inner, then we are done. Otherwise, since $|Z_{4}(G)/Z(G)|=p^{5}$, it follows from Remark \ref{tobeinner} that $[Z_{4}(G),G]=\Omega_{1}(\langle Z_{2}(G),k\rangle )$. 

\vspace{4pt}

Summarizing what we have obtained so far, we notice that  either $G$ has a non-inner automorphism of order $p$ or  
\[\Omega_{1}(\langle Z_{2}(G), t \rangle)  = [Z_{4}(G),G] = \Omega_{1}(\langle Z_{2}(G), k \rangle).\]
 On the other hand, note that
$$Z_{2}(G) = \langle Z_{2}(G), t \rangle  \cap \langle Z_{2}(G) , k \rangle \geq \Omega_{1}(\langle Z_{2}(G), t \rangle)  \cap \Omega_{1}(\langle Z_{2}(G) , k \rangle).$$ 
In the latter case above, this implies that $\Omega_1(Z_{2}(G))$, which is of order $p^2$, contains the subgroup  $\Omega_{1}(\langle Z_{2}(G), t \rangle)$, which is of order $p^3$. This contradictions completes the proof of the theorem.
\end{proof}

As a result, in case there exists a counterexample to the conjecture for a finite $p$-group $G$ of coclass $3$, where $p$ is an odd prime, then $p$ must be equal to $3$.  Our techniques have certain limitations, and do not work for $p = 3$. In this situation, firstly, there are no quotients of maximal class, nilpotency class $3$ and exponent $3$, and secondly it is not possible to have relations defined by commutators of weight $3$ that vanish in a module that is not contained in the third center. It will be interesting to see that the conjecture holds true for finite $3$-groups of coclass $3$.

\vspace{4pt}

The following result  now completes the proof of Main Theorem in its full generality.
\begin{thm}\label{marco2}
Let $G$ be a finite $2$-group of coclass $3$ satisfying Hypothesis $B$ and having nilpotency class at least $4$. Then $G$ admits a non-inner automorphism of order $2$.
\proof

With the given conditions, we want to construct a special derivation, which leads to a non-inner automorphism of $G$ of order $2$. By Remark \ref{remark1}, we can assume that $G$ is a two generator finite $2$-group, $Z_{2}(G)/Z(G) \cong C_{2} \times C_{2}$ and $|Z_{3}(G)/Z_{2}(G)| \in \{2,4\}$. Furthermore, by Lemma \ref{specialheadgroup} we know that $G$ admits a quotient $G/N$ isomorphic to the dihedral group of order $8$. Moreover, by the proof of Lemma \ref{specialheadgroup}, we observe that $Z_{c-2}(G) \leq N < G^{2}$, and that $\Omega_{1}(Z_{2}(G)) \leq Z(N)$. Let us fix a non-central element $u \in \Omega_{1}(Z_{2}(G))$. Then  $C_{G}(u)$ is maximal in $G$, and we can choose two generators $x,y$  of $G$ such that $[u,y]=1$ and $[u,x] \ne 1$. Notice that, for the dihedral quotient $G/N$, either $y^{2}N = \overline{1}$ or $y^{2}N \neq \overline{1}$. If $y^{2}N = \overline{1}$, then it is possible to choose $x$ such that $xN$ has order $4$ in
  $G/N$. If $y^{2}N \neq \overline{1}$, then we choose $x$ such that $xN$ has order $2$ in $G/N$. Thus, it is possible to use $xN$ and $yN$ as generators of $G/N$ and to assume that one of them has order $2$ and the other one has order $4$.  Now let us define the map $\delta:G/N \rightarrow \Omega_{1}(Z_{2}(G))$ such that $\delta(\overline{x}^{i}\overline{y}^{j}\overline{w}^{k})= u^{j}z^{k}$, where $z=[x,u] \in \Omega_{1}(Z(G))$, $[\overline{x}, \overline{y}]={\overline{w}}$ and $1 \le i,j,k \le 2$. An easy calculation shows that the map $\delta$ is well-defined. Indeed, $\delta$ is a derivation. In fact, 

\begin{eqnarray*}
\delta(\overline{x}^{i}\overline{y}^{j}\overline{w}^{k}\overline{x}^{i'}\overline{y}^{j'}\overline{w}^{k'})
&=&\delta(\overline{x}^{i}\overline{y}^{j}\overline{x}^{i'}\overline{y}^{j'} \overline{w}^{k+k'}) =\delta(\overline{x}^{i+i'}\overline{y}^{j} [{\overline{y}}^{j},{\overline{x}}^{i'}]\overline{y}^{j'}\overline{w}^{k+k'})\\
&=&\delta(\overline{x}^{i+i'}\overline{y}^{j+j'}\overline{w}^{k+k'-ji'})=
u^{j+j'}z^{k+k'-ji'}\\
&=&u^{j+j'}z^{-ji'}z^{k+k'} = u^{j+j'}[u^{j},x^{i'}]z^{k+k'}\\
&=&u^{j}[u^{j},x^{i'}]z^{k}  u^{j'}  z^{k'} = (u^{j})^{x^{i'}}z^{k}  u^{j'}  z^{k'}\\
&=& (u^{j}z^{k})^{x^{i'}}  u^{j'}  z^{k'} = (u^{j}z^{k})^{x^{i'}y^{j'}w^{k'}}  u^{j'}  z^{k'}\\
&=&\delta(\overline{x}^{i}\overline{y}^{j}\overline{w}^{k})^{\overline{x}^{i'}\overline{y}^{j'}
\overline{w}^{k'}}\delta(\overline{x}^{i'}\overline{y}^{j'}\overline{w}^{k'}).
\end{eqnarray*}
Using Remark \ref{naturalwaylift} this derivation can be lifted to a derivation $\delta':G \rightarrow \Omega_{1}(Z_{2}(G))$ by the law $\delta'(g)=\delta(gN)$, for all $g \in G$.  Since $\Omega_{1}(Z_{2}(G)) \leq Z(N)$  (so $\delta'(\Omega_{1}(Z_{2}(G)))=1$),  $\delta'$ extends to an automorphism $\phi$ of $G$ of order $2$ defined by  $\phi(g)=g\delta'(g)$, for all $g \in G$. This automorphim $\phi$ does not centralize $\Phi(G)$. To finish with the proof, it only remains  to prove that $\phi$ is non-inner.  Contrarily assume that $\phi$ is inner. Then there exists an element  $k \in Z_{3}(G) -  Z_{2}(G)$ such that $\phi(g)=g^{k}$, for all $g \in G$. As a consequence, the following relations hold:  $[y,k]=u$, and $[x,k]=1$. Moreover, either $[x^{2},k] \ne 1 $ or $[y^{2},k] \ne 1$. In the former case,  we have  $$1 \ne [k,x^{2}]=[k,x]^{2}[[k,x],x]=1,$$
 which is absurd. Thus $[y^{2},k] \ne 1$. However,  in this case, $$1\ne [k,y^{2}]=[k,y]^{2}[[k,y],
 y]=(u^{-1})^{2}[u^{-1},y]=1,$$
  which cannot happen again.  The proof of the theorem is now complete.

\endproof
\end{thm}  

\section{Acknowledgements}
The first author would like to thank the Department of Mathematics at  
the University of the Basque Country for its excellent hospitality  
while part of this paper was being written; on the other hand, the  
first two authors also wish to thank Professors Norberto Gavioli, Gustavo A. Fern\'andez Alcober and  
Carlo Maria Scoppola for their suggestions.

\bibliography{articles}

\bigskip
\bigskip

\end{document}